\documentclass[12pt]{imsart}
\usepackage{iftex}

\ifPDFTeX 
  \usepackage[utf8]{inputenc}
  \usepackage[T1]{fontenc}
  \usepackage{lmodern}
\fi

\usepackage[margin=1in]{geometry}
\usepackage{graphicx,xcolor,mathrsfs,bbm}

\usepackage[toc]{appendix}
\RequirePackage[numbers]{natbib}
\usepackage{amsmath,amssymb,amsthm,bm,latexsym,mathtools}
\usepackage{amsfonts}
\usepackage{dsfont}
\usepackage{orcidlink}
\usepackage{comment}
\usepackage{enumitem}
\usepackage{url}
\usepackage{booktabs}
\usepackage{hyperref}
\hypersetup{colorlinks=true,citecolor=blue,linktocpage=true}
\usepackage{cleveref}

\startlocaldefs

\makeatletter
\newcommand{\namedlabel}[2]{\begingroup #2\def\@currentlabel{\textnormal{#2}}%
  \phantomsection\label{#1}\endgroup}
\makeatother
  
\setlist[enumerate,1]{label=(\roman*)}
\setlist[enumerate,2]{label=(\alph*)}

\def\E{\mathbb{E}}
\def\P{\mathbb{P}}
\def\N{\mathbb{N}}

\def\R{\mathbb{R}}

\def\T{\mathbb T}
\def\Fbar{\overline F}

\AtBeginDocument{}

\newtheorem{dummy}{***}[section]
\newtheorem{theorem}[dummy]{Theorem}

\newtheorem{lemma}[dummy]{Lemma}
\newtheorem{remark}[dummy]{Remark}

\AddToHook{env/theorem/begin}{\crefalias{dummy}{theorem}}
\AddToHook{env/example/begin}{\crefalias{dummy}{example}}
\AddToHook{env/corollary/begin}{\crefalias{dummy}{corollary}}
\AddToHook{env/definition/begin}{\crefalias{dummy}{definition}}
\AddToHook{env/assumption/begin}{\crefalias{dummy}{assumption}}
\AddToHook{env/lemma/begin}{\crefalias{dummy}{lemma}}
\AddToHook{env/remark/begin}{\crefalias{dummy}{remark}}
\AddToHook{env/proposition/begin}{\crefalias{dummy}{proposition}}
\crefname{theorem}{Theorem}{Theorems}
\crefname{example}{Example}{Examples}
\crefname{corollary}{Corollary}{Corollaries}
\crefname{definition}{Definition}{Definitions}
\crefname{assumption}{Assumption}{Assumptions}
\crefname{lemma}{Lemma}{Lemmas}
\crefname{remark}{Remark}{Remarks}
\crefname{proposition}{Proposition}{Propositions}

\numberwithin{equation}{section}  

\endlocaldefs

\begin{document}

\begin{frontmatter}
\title{\Large A Sharp Phase Transition for Killed Branching Random Walks with Heavy-Tailed Displacements}
\runtitle{Phase Transition for Killed Branching Random Walks}
\runauthor{R. J. Samanta}

\author[UCL]{%
\fnms{Ramkrishna Jyoti} \snm{Samanta}%
\ead[label=e1]{ramkrishna.samanta.24@ucl.ac.uk}}

\address[UCL]{%
Department of Statistical Science,
University College London,
Gower Street, London WC1E 6BT, UK}
\footnotetext{E-mail: \texttt{ramkrishna.samanta.24@ucl.ac.uk}}

\begin{abstract}
We study the survival of a branching random walk in a supercritical Galton-Watson tree subject to a deterministic, superlinear killing barrier with heavy-tailed displacements. For a fixed $m>0$, we consider a family of such barriers
and kill particles whose ancestral paths fall below the barrier. Under some standing assumptions, we establish a sharp phase transition at $m=4$: the process becomes extinct almost surely for $m<4$, while survives with positive probability for $m>4$. Notably, a natural first-moment heuristic suggests the threshold $m=2$; the true critical value $4$ arises from a deterministic constraint along infinite surviving rays. At the critical value $m=4$, we provide examples showing that the same standing assumptions do not determine the survival behavior. The proofs are based on analysis of infinite surviving rays and embedded trees. 
\end{abstract}
\end{frontmatter}
\maketitle
\section{Introduction}
Let $\T$ be a Galton-Watson tree, with offspring distribution satisfying 
\begin{equation}\label{eq:GW1}
\mu=\E[\xi]\in(1,\infty).\tag{GW1}
\end{equation} 
It is well known that such a process has a positive probability of survival. We identify this tree with a set of vertices or nodes $v\in \T$ and a distinguished vertex $\rho$ denoting its root. For vertices $v, w\in \T$, we say that $w \preceq v$ if $w$ is an ancestor of $v$.  Thus, 
$$\{w\in \T: w \preceq v\}$$ denotes its ancestral lineage, including $v$ itself and the root $\rho.$ To each vertex $v \in \T \setminus\{\rho\},$ we associate an iid random variable $X_v,$ which we interpret as displacement of the edge connecting $v$ and its parent. Set $S(\rho):=0$ and the position of $v$ is thus given by summing up the displacement along its ancestral lineage. In particular,  
$$S(v) :=  \sum_{w\preceq v} X_w.$$
The collection $\{S(v):v \in \T\}$ is the \textit{branching random walk} associated with $\T.$ We denote $|v|$ to be the generation of the vertex $v \in \T$ and if 
$$\rho=v_0 \prec v_1\prec\ldots\prec v_n = v,$$ we define 
$$X_j(v):= X_{v_j}, \quad S_j(v):=S(v_j), \quad \forall\; j\leq |v|=n.$$ A natural question that arises is, does spatial restriction affect the survival, even when the underlying Galton-Watson process has a positive probability of survival. In particular, we are motivated by the problem considered in \citet{BigginsLubachevskyShwartzWeiss1991}, where particles of a branching random walk are removed when they cross a prescribed barrier. Let $\{b_n\}_{n\geq 0}$ be a deterministic sequence. For a vertex $v \in \T$ with $|v|=n$, we say $v$ is \textit{alive} if 
$$S_j(v)\geq b_j,
\quad \forall\; 1\leq j\leq n$$ and call the sequence as \textit{killing barrier}. 

\citet{BigginsLubachevskyShwartzWeiss1991} obtained a criterion determining whether a branching random walk can survive in the presence of a linear barrier. A corresponding question for branching Brownian motion was studied earlier in \citet{Kesten1978}. For standard binary branching Brownian motion with absorption at the origin, he identified a drift criteria determining almost-sure extinction and survival with positive probability. In the usual normalization, there is a positive probability of survival if and only if the drift of the particles is greater than $-\sqrt{2}$. More involved analysis of the absorption barrier was done in \citet{Jaffuel2012}, where instead of taking a linear barrier, they consider a barrier $b_n = an^{1/3},$ and obtained criteria on $a$ for survival. A similar result was shown to hold in \citet{LiuZhang2019}, where the associated random walk is in the domain of attraction of an $\alpha$-stable law. These works also derive motivation from various studies of extreme particles in the scenario with no spatial killing. Under suitable exponential moment assumptions on the displacement distribution, \citet{Hammersley1974} showed that the maximal position of a supercritical branching random walk grows linearly. More precisely, if
$M_n:=\max_{|v|=n} S(v),$
then, on the event of non-extinction,
$M_n/n\rightarrow \gamma$
almost surely, where $\gamma$ is a deterministic constant depending on the offspring and displacement distributions. This was further considered under a  killed scenario by \citet{GantertHuShi2011},  where the barrier was chosen to be linear with slope $\gamma-\varepsilon$, and studied a finer question of how does the probability of survival behave as $\varepsilon \downarrow 0.$ Questions of persistence of killed branching Brownian motion was considered in the subcritical case by \citet{AidekonJaffuel2011}. 

A similar path is followed in the scenarios where the exponential moment assumption on the displacement distribution of the associated random walk does not hold. \citet{Durrett1983} considers a supercritical branching random walk, where displacement have regularly varying tails. He showed that $M_n$ grows on an exponential scale and that there is a normalizing sequence $c_n$ such that $M_n/c_n$ converges in distribution to a suitable limit. Further, \citet{Gantert2000} studied a similar question with the case of stretched exponential distribution, and showed the existence of a sequence $\psi_n$ such that $M_n/\psi_n$ converges to a deterministic limit. 

The study of maxima in Branching Brownian motion has been further developed in \citet{AddarioBerryReed2009} and \citet{aidekon2013convergence}. A finer understanding of the extremal particles have also been considered through
the convergence of extremal point processes. In the light-tailed setting,
such results have been established for branching Brownian motion; see, for
instance, \citet{ArguinBovierKistler2013,AidekonBerestyckiBrunetShi2013}, and
for branching random walks in \citet{Madaule2017}. Corresponding results have been obtained in heavy-tailed
settings. For branching random walks with regularly varying displacements,
\citet{BhattacharyaHazraRoy2017,BhattacharyaHazraRoy2018} established
convergence of the extremal point process. More recently, related questions
for stretched exponential displacements have been considered in
\citet{DyszewskiGantertExtremal} and for log-slowly varying tails in \citet{BhattacharyaDyszewskiGantertPalmowski2025}. 

Existing results for killed branching random walks therefore, have dealt with linear or near-linear barriers in the cases where the step distribution is light-tailed, whereas the heavy-tailed branching-random-walk literature has mostly focused on extremal particles without any spatial constraint. Here, we address the corresponding question for heavy-tailed branching-random-walks with killing barrier.

\vspace{0.2in}

In this context, define $\Fbar(x) := \P(X\geq x)$ and for fixed $m>0$, we consider the barrier of the form 
$$b_0^{(m)} := 0 \quad\text{ and }\quad b_n^{(m)}:= a(n/m).$$ 

We work under the following standing assumptions.
\begin{description}[leftmargin=3.2em,labelwidth=2.7em,itemsep=1em]

\item[\normalfont{\rm(A1)}]
The function $a:(0,\infty)\to(0,\infty)$ is strictly increasing and
$a(t)/t\rightarrow\infty,$ as $t\to\infty.$

\item[\normalfont{\rm(A2)}]
$\log_\mu \Fbar(a(t))
=
-t+o(t),
\quad t\to\infty.
$

\item[\normalfont{\rm(A3)}]
For every fixed $q>0$,
$\log_\mu
\Fbar(a(t/q)/t)
=
-t/q+o(t),
\quad t\to\infty.$

\end{description}

For the left tail, we assume
\begin{equation}
\sum_{n=2}^{\infty}
n\,\P\left(
X<-\frac{a((n-1)/m)}{n}
\right)
<\infty.
\tag{$L_m$}
\label{eq:Lm}
\end{equation}

Assumption \textnormal{(A1)} makes sure that the killing barrier $b_n^{(m)}$ is superlinear in $n.$ By \textnormal{(A2)}, $\overline F(a(t))=\mu^{-t+o(t)}$ and since the expected number of particles of the underlying Galton-Watson process at generation $n$ is $\mu^n$, the level $a(n)$ is the scale where the offspring growth of the tree and the probability of such  a displacement balance on the exponential scale, since $\mu^n\overline F(a(n))=\mu^{o(n)}$. This carries motivation from the formulation in \citet{Durrett1983}. \textnormal{(A3)} is used in the proof of extinction in order to upper bound the persistence probability of paths above the barrier. We further note that for many standard heavy-tailed distributions, including regularly varying and lognormal-type tails, \textnormal{(A2)} would also imply \textnormal{(A3)}, as will be discussed in the subsequent remark. The condition ($L_m$) is usually satisfied whenever the distribution has light-tailed left tails or exponential moments of the left tails exist and used along with \textnormal{(A2)} for proving the survival part.

\begin{remark}
The above assumptions are satisfied by several standard classes of
heavy-tailed distributions. We provide two examples.

First, suppose that the right and left tails are of lognormal type, that is, as $x\to\infty$,
$$
\P(X\geq x)
=
\exp\left\{
-\frac{(\log x-\mu_{+})^2}{2\sigma_{+}^2}
+o((\log x)^2)
\right\},
$$
and
$$
\P(X\leq -x)
=
\exp\left\{
-\frac{(\log x-\mu_{-})^2}{2\sigma_{-}^2}
+o((\log x)^2)
\right\},
$$
for some $\mu_{+},\mu_{-}\in\R$ and
$\sigma_{+},\sigma_{-}>0$. Then one may take
$$
a(t)
=
\exp\left\{
\mu_{+}+\sigma_{+}\sqrt{2(\log\mu)t}
\right\}.
$$
Similarly, if
$$
\P(X\geq x)=x^{-\alpha_{+}}L_{+}(x),
\qquad
\P(X\leq -x)=x^{-\alpha_{-}}L_{-}(x),
$$
where $\alpha_{+},\alpha_{-}>0$ and $L_{+},L_{-}$ are slowly varying functions, then
one can consider
$$
a(t)=\mu^{t/\alpha_{+}}.
$$
Similarly, various combination of these types of distributions also satisfy the above assumptions.
\end{remark}
Our main result is the following.
\begin{theorem}\label{main-thm}
Assume \eqref{eq:GW1}, \textnormal{(A1)--(A3)}.
\begin{enumerate}[label=\textnormal{(\roman*)}]
\item
If $m<4$, then the killed branching random walk becomes extinct almost surely.
\item
If $m>4$ and in addition, \eqref{eq:Lm} holds, then the killed branching random walk survives
with positive probability.
\end{enumerate}
\end{theorem}

 The result is particularly interesting in the sense that, a natural first-moment heuristic suggests the threshold $m=2$; the true critical value $4$ arises from a deterministic constraint along any infinite surviving ray. We note that a regular first-moment argument is often enough in the classical light-tailed displacement cases. For example, \citet{BigginsLubachevskyShwartzWeiss1991} and \citet{Jaffuel2012} reduces the expected number of particles, whose ancestral path satisfies certain conditions to a question about a random walk through a many-to-one lemma and then apply several estimates of the  random walk. Although the standard many-to-one or the first moment identity can still be used in our setting, the exponential change-of-measure arguments used in light-tailed settings are not applicable, since exponential moments of the displacement distribution need not exist.

To see this heuristically, we restrict to the case where the associated random walk has the step distribution which is supported only on the positive side. We start by isolating particles at geometrically
increasing times $t_K=2^{K-1}$, and let $\Gamma_K$ be the number of such particles. We  require these  particles to make jumps
which are at the scale of the barrier at the next geometric generation. These jumps give the particle enough buffer to remain above the barrier until
the next large jump occurs and remain alive according to the above criteria. Indeed, if these particles have a positive probability of survival, this would imply the same for the original one. We further note that, the descendants of the particles, counted in $\Gamma_{K-1}$ are all alive until $N=t_K$. And, we obtain 
\begin{align*}
    &\log_\mu \E\left[\text{\# alive descendants at generation } N \text{ of particles counted in } \Gamma_{K-1}\right]
\hspace{.5in}\\&\hspace{3.5in}=
\left(1-\frac{2}{m}\right)N+o(N).
\end{align*}
Thus a first-moment calculation suggests the value $m=2$: as soon as $m>2$, the expected number of alive particles grows exponentially, while for $m<2$,  $\Gamma_K$ goes extinct almost surely. Therefore, the  analysis is insufficient to argue that $m=4$ is the critical value.

To prove that the probability of survival for $m>4$, we choose an embedded tree where particles consistently perform multiple sufficiently large jumps and use second moment method. A similar approach appears in the proof of the lower bound for the maximal particle in \citet{Gantert2000}. 

However, to see $m=4$ is indeed the critical one, we analyze the infinite surviving rays. \citet{GantertHuShi2011} studied the survival probability through the existence of an infinite surviving ray. However, the techniques of analyzing the rays are not applicable here, already discussed. In our case, if an infinite surviving ray, say $$\rho=x_0\prec x_1\prec x_2\prec\ldots\;$$ exists, consider the \textit{successive new maxima times} along the ray by defining
$$\tau_1 = 1, \qquad\tau_{k+1} = \min \{j>\tau_k:X_j>X_{\tau_k}\}$$
and we show that there must be infinite number of times such that 
$$1=\tau_1<\tau_2<\cdots.$$
These jumps support the trajectory until a later new maxima time of the increment occurs. We then make use of the deterministic constraint
\[
    \limsup_{k\to\infty}
    \frac{\tau_2+\cdots+\tau_{k+1}}{\tau_k}
    \geq 4.
\]
One might get an idea of this deterministic result by considering geometric sequences. This lead us to construct a sequence of events which if do not occur infinitely often will imply extinction, and a Borel-Cantelli argument concludes the proof.

The remainder of the paper is organized as follows. In Section $2$, we prove the main theorem. We first establish almost-sure extinction for $m<4$ and then prove survival with positive probability for $m>4$ with finite variance restriction on the offspring distribution which is later removed by a truncation argument. Section $3$ handles the critical case $m=4$, where we discuss two distributions satisfying the standing assumptions, but differs in the survival behavior. Finally, Section $4$ contains the proofs of the technical lemmas used in the preceding sections.

\section{Proof of Theorem \ref{main-thm}}
\subsection{Proof of extinction when $m<4$}
Here, we prove the extinction part of the main result. Fix $m<4$ and let $S(\rho)\geq b_0$. If the killed branching random walk survives forever, there must exist a sequence of alive vertices forming an infinite ray. That is, there must exist an infinite genealogical ray
$$\rho=x_0\prec x_1\prec x_2\prec\ldots\;$$ such that $$S(x_n)\geq b_n= a(n/m),  \qquad \forall n\geq 0.$$
Therefore, it is enough to show that almost surely, no such infinite ray exists.  

To do so, we consider the behavior of a typical infinite path that would survive. Along such a ray, we write 
$$X_j:=X_{x_j}, \qquad j\geq 1.$$
We first observe that $\{X_j\}_{j\geq 1}$ must be unbounded. Indeed, if there exists $C>0$ such that $X_j\leq C,  \forall j$, this would imply 
$$S(x_n)\leq nC$$ and this contradicts \textnormal{(A1)}. We may therefore define a sequence of \textit{successive new maxima times} along the ray by
$$\tau_1 = 1, \qquad\tau_{k+1} = \min \{j>\tau_k:X_j>X_{\tau_k}\}.$$
Since the increments along the surviving ray are unbounded above, each
$\tau_k<\infty$ for all $k$, almost surely. For $s\geq2$, define
$$h_s:=\frac{b_{s-1}}{s-1}.$$
By the definition of $\tau_{i+1}$, there is no increment before generation
$\tau_{i+1}$ that exceeds $X_{\tau_i}$. Therefore
$$
S(x_{\tau_{i+1}-1})
=
\sum_{j=1}^{\tau_{i+1}-1}X_j
\leq
(\tau_{i+1}-1)X_{\tau_i}.$$
Also, since the ray remains above the killing barrier implies
$S(x_{\tau_{i+1}-1})
\geq
b_{\tau_{i+1}-1},$
it follows that
\begin{equation}\label{new-max-lower-bnd}
X_{\tau_i}
\geq
\frac{b_{\tau_{i+1}-1}}{\tau_{i+1}-1}
=
h_{\tau_{i+1}},
\qquad i\geq1.
\end{equation}
We now note the following deterministic lemma, whose proof is given in Section $4$.
\begin{lemma}\label{lem:four}
Let
$$1=r_1<r_2<\cdots$$
be a deterministic strictly increasing sequence of positive integers. Then
$$\limsup_{k\to\infty}
\frac{r_2+\cdots+r_{k+1}}{r_k}
\geq4.$$
Moreover, the constant $4$ is sharp.
\end{lemma}
Fix $m<M<4$. From the above deterministic lemma, it is clear that 
along an infinite ray, its corresponding $\{\tau_k\}_{k\geq 1}$ must satisfy 
$$\tau_2+\ldots+\tau_{k+1}\geq M\tau_k$$ for infinitely many $k$'s. 
Now, we introduce the following sequence of events for $t\geq 2,$ where 
\begin{equation}\label{def-E_t-1}
E_t
=
\left\{
\begin{array}{l}
\exists\, 1=r_1<\cdots<r_k=t<r_{k+1}\ \text{and } v \text{ with } |v|=t,\\
X_{r_i}(v)\geq h_{r_{i+1}},\qquad 1\leq i\leq k,\\
r_2+\cdots+r_{k+1}\geq Mt
\end{array}
\right\}.
\end{equation} 
Observe that on the event that an infinite ray exists, $\{E_t\}$ must occur infinitely often, almost surely. Indeed, considering the sequence $\tau_k$ along the infinite ray, for every $k$ satisfying
$$\tau_2+\cdots+\tau_{k+1}\geq M\tau_k,$$
we may take
$$
t=\tau_k,\qquad r_i=\tau_i,\qquad v=x_{\tau_k}.
$$
Then, \eqref{new-max-lower-bnd} shows that all the conditions in the
definition of $E_t$ are satisfied. Therefore, it is enough to show that $E_t$ occurs only finitely often, almost surely. The next lemma proves this fact. 
\begin{lemma}
\label{Et-summable}
Let $m<M<4$, and let $\{E_t\}_{t\geq2}$ be defined in \eqref{def-E_t-1}. Then
$$\sum_{t=2}^{\infty}\mathbb P(E_t)<\infty.$$
\end{lemma}
\begin{proof}    
Define $p_k := \P(X\geq h_k).$ Fixing a sequence $\{r_2,\ldots,r_{k+1}\}$ and a vertex $v$ such that $|v|=t,$ we note that 
$$\P(X_{r_i}(v)\geq h_{r_{i+1}} \text{ for all }1\leq i\leq k)=\prod_{i=1}^k p_{r_{i+1}}.$$ Therefore, using the first moment with the vertices at generation $t$ and markov inequality, we get 
$$\P(E_t) \leq
\mu^t
\sum_{\substack{
k\geq2,\,
1=r_1<\cdots<r_k=t<r_{k+1}\\
r_2+\cdots+r_{k+1}\geq Mt
}}
\prod_{i=1}^{k}p_{r_{i+1}}.$$
We now use a further upper bound on the above estimate of the right hand side in the following way for some fixed $\theta>0$.
\begin{align}
\P(E_t) \leq
\mu^t
\sum_{\substack{
A\subset\{2,3,\ldots\}\text{ finite}\\
\sum_{s\in A}s\geq Mt
}}
\prod_{s\in A}p_s&\leq \mu^{(1-\theta M)t}\sum_{\substack{
A\subset\{2,3,\ldots\}\text{ finite}\\
}}
\prod_{s\in A}\mu^{\theta s}p_s
\nonumber\\
&= \mu^{(1-\theta M)t}
\prod_{s=2}^{\infty}
\left(1+p_s\mu^{\theta s}\right).
\end{align}
Finally, choosing $1/M<\theta<1/m-\varepsilon$ and using \textnormal{(A3)}, it follows that, for all sufficiently large $s$,
$$
\mu^{\theta s}p_s
\leq
\mu^{-(1/m-\varepsilon-\theta)s}.
$$
Consequently, $\sum_{s=2}^{\infty}\mu^{\theta s}p_s<\infty,$
and therefore
$C_\theta:=
\prod_{s=2}^{\infty}
\left(1+\mu^{\theta s}p_s\right)
<\infty.$
Finally, since $\theta M>1$, we have
$$
\P(E_t)
\leq
C_\theta\mu^{-(\theta M-1)t}.
$$
The right-hand side is summable in $t$.
\end{proof}
This completes the proof.

\subsection{Proof of survival when $m>4$}
In this section, we prove the survival part of the main theorem. Throughout
the section, we fix $m>4$ and assume that $(L_m)$ holds. We
first work under the additional assumption that
$$\E[\xi^2]<\infty.$$ 
Let $t_0 =0$ and
$t_k:=2^{k-1},\;\text{for } k\geq1.$
We construct a tree $\T_2$ embedded in the original Galton-Watson tree
$\T$. The root of $\T_2$ is $\rho$, and the vertices at level $k$ of
$\T_2$ are vertices of $\T$ at generation $t_k$. The construction has some inspiration from \citet{Gantert2000}.

More precisely, the vertices $w_k$ at level $k\geq 1$ of $\T_2$ are required to satisfy the following: 
\begin{equation}\label{first-constraint-embedded}
    w_0 =\rho, \quad |w_k| = t_k,\quad X_{w_k}\geq b_{t_{k+1}}=a(t_{k+1}/m)
\end{equation}
and 
\begin{equation}\label{second-constraint-embedded}
    S(u)\geq S(w_{k})- b_{t_k-1}, \qquad \forall\; w_{k}\prec u \prec w_{k+1}.
\end{equation}

The above two restrictions make sure that the vertices in $\T_2$ attain very high heights early enough so that the buffer coming from \eqref{second-constraint-embedded} can keep the vertices in $\T_2$ alive. 

In particular, we claim that survival of $\T_2$ implies survival of $\T$. Indeed, suppose $\T_2$ contains an infinite ray
\begin{equation}
    \rho=w_0\prec w_1\prec w_2\prec\cdots.
\end{equation}
Recall that $S(\rho)=0$ and thus, we now argue inductively over the geometric blocks. Suppose that the path is
alive up to generation $t_k-1$, which implies that the parent of $w_k$ satisfies
\begin{equation}
S(\operatorname{par}(w_k))\geq b_{t_k-1}.
\end{equation}
Using \eqref{first-constraint-embedded} and \eqref{second-constraint-embedded}, we immediately obtain that
$$S(w_k)= S(\operatorname{par}(w_k)) +X_{w_k} \geq b_{t_k-1}+b_{t_{k+1}}$$ and therefore
$$S(u)>S(w_{k})-b_{t_k-1}\geq b_{t_{k+1}}\geq b_{|u|},\quad \forall\; w_{k}\preceq u\prec w_{k+1}.$$ 
This proves our claim. Hence, it is enough to show that $\T_2$ has a positive probability of survival. 

We begin by denoting $Z_K$ to be the number of vertices in the $K$-th generation of $\T_2$. In what follows, we aim to prove that 
\begin{align}\label{extinction}
    \inf_{K\geq 1}\P(Z_K>0)>0.
\end{align}
Since $\{Z_K>0\}$ is decreasing in $K$, this would imply
$$\P(\{\T_2 \text{ survives}\})=\lim_{K\rightarrow \infty}\P(Z_K>0)>0.$$
To obtain this, we use second moment method argument. We denote $Y_K$ to be the subset of the vertices in $\T_2$ which only \eqref{first-constraint-embedded} holds for all $1\leq k\leq K$. Write $$\pi_k:=\P(X\geq b_{t_{k+1}})$$ and hence
$$\E[Y_K]=\mu^{t_K}\prod_{k=1}^{K}\pi_{k}.$$
It is clear that $\E[Z_K]\leq \E[Y_K]$, but we show a reverse comparison in terms of the first moment in the following lemma.
\begin{lemma}
There exists a constant $q_\infty>0$ such that, for every $K\geq1$,
$$
\E[Z_K]\geq q_\infty \E[Y_K].
$$
\end{lemma}

\begin{proof}
Fix a vertex $v$ with $|v|=t_K$. For $k\geq1$, let
$$
T_k(v):=
\left\{
\min_{0< r< t_{k+1}-t_k}
\sum_{j=1}^{r}X_{t_k+j}(v)
\geq -b_{t_k-1}
\right\}.
$$
Let 
$q_k:=\P(T_k).$ The event that $v$ is counted by
$Y_K$ is
$$
\bigcap_{k=1}^{K}
\{X_{t_k}(v)\geq b_{t_{k+1}}\},
$$
whereas $v$ is counted by $Z_K$ if, in addition,
$\bigcap_{k=1}^{K-1}T_k(v)$
occurs. Since along a fixed genealogical path the edge increments are independent, we have 
$$
\P(v\text{ is counted in }Z_K)
=
\left(\prod_{k=1}^{K}\pi_k\right)
\left(\prod_{k=1}^{K-1}q_k\right).
$$
The first moment method gives
$\E[Z_K]
=
\E[Y_K]\prod_{k=1}^{K-1}q_k.$
Next, we note that
$$
X_{t_k+j}(v)\geq
-\frac{b_{t_k-1}}{t_{k+1}-t_k} = -\frac{b_{t_k-1}}{t_k},
\qquad 1\leq j\leq t_{k+1}-t_k-1,
$$
implies $T_k(v)$. Therefore, by a union bound,
$$
1-q_k
\leq
t_k\P\left(
X<-\frac{b_{t_k-1}}{t_k}
\right)
=
t_k\P\left(
X<-\frac{a((t_k-1)/m)}{t_k}
\right).
$$
Condition $(L_m)$ implies since $q_k>0$ for every $k$,
$q_\infty:=\prod_{k=1}^{\infty}q_k>0.$
Hence, proved.
\end{proof}
\vspace{0.2in}
Therefore, combining the above lemma with $\E[Z_K^2]\leq \E[Y_K^2]$ and Paley-zygmund inequality, the following holds:
\begin{align}\label{paley}
    \frac{1}{\P(Z_K>0)}\leq \frac{\E[Z_K^2]}{(\E[Z_K])^2}\leq \frac{1}{q_\infty^2}\frac{\E[Y_K^2]}{(\E[Y_K])^2}.
\end{align}
Therefore, it suffices to show that 
\begin{align}\label{paley-Y}
\sup_{K\geq 1}    \frac{\E[Y_K^2]}{(\E[Y_K])^2}<\infty.
\end{align}
Now, we prove this claim. Set $N:=t_K$ and
$P_K:=\prod_{k=1}^{K}\pi_k.$
We use the standard genealogical second-moment decomposition; see, for instance, \citet{HarrisRoberts}.
Let
$$
\nu:=\E[\xi(\xi-1)]<\infty.
$$
We first note that for $0\leq t\leq N-1$, the expected number of ordered pairs of distinct
vertices at generation $N$ whose most recent common ancestor is at
generation $t$ is
$$
\nu\mu^{2N-t-2}.
$$
Indeed, there are on average $\mu^t$ vertices at generation $t$, there
are on average $\nu$ ordered pairs of distinct children of such a
vertex, and each of the two child subtrees has expected
$\mu^{N-t-1}$ descendants at generation $N$.

If two such vertices have most recent common ancestor at generation
$t$, then the distinguished jumps at times $t_k\leq t$ are shared,
whereas those after time $t$ occur on their independent genealogical paths.
Therefore,
$$
\P(\text{both vertices are counted by }Y_K)
=
\frac{P_K^2}
{\prod_{t_k\leq t}\pi_k}.
$$
Next, separating the distinct pairs from the alike pairs of vertices, we obtain
\begin{align*}
\E[Y_K^2]
&=
\nu\sum_{t=0}^{N-1}
\mu^{2N-t-2}
\frac{P_K^2}
{\prod_{t_k\leq t}\pi_k}
+\mu^N P_K.
\end{align*}
Since $(\E[Y_K])^2=\mu^{2N}P_K^2$, it follows that for some constant $C>0,$ we have
\begin{align}\label{second-moment-decompsotion}
\frac{\E[Y_K^2]}{(\E[Y_K])^2}
&=
\frac{\nu}{\mu^2}
\sum_{t=0}^{N-1}
\frac{\mu^{-t}}
{\prod_{t_k\leq t}\pi_k}
+
\frac{\mu^{-N}}{P_K}  \leq
C
\sum_{t=0}^{N}
\frac{\mu^{-t}}
{\prod_{t_k\leq t}\pi_k}.
\end{align}
Choose $\varepsilon>0$ such that
$4(1/m+\varepsilon)<1.$
By assumption \textnormal{(A2)}, for all sufficiently large $k$,
$$
\pi_k
=
\Fbar\left(a(t_{k+1}/m)\right)
\geq
\mu^{-(1/m+\varepsilon)t_{k+1}}.
$$
Hence, there
exists $c_\varepsilon>0$ such that, for every $j\geq1$,
$$
\prod_{k=1}^{j}\pi_k
\geq
c_\varepsilon
\mu^{-(1/m+\varepsilon)
\sum_{k=1}^{j}t_{k+1}}.
$$
For $1\leq t\leq N$, let $j$ be such that
$t_j\leq t<t_{j+1}.$ Then
$\prod_{t_k\leq t}\pi_k
=\prod_{k=1}^{j}\pi_k,$
and since $t_k=2^{k-1}$,
$\sum_{k=1}^{j}t_{k+1}
=
4t_{j}-2
<4t.$
Hence
$$
\frac{\mu^{-t}}
{\prod_{t_k\leq t}\pi_k}
\leq
c_\varepsilon^{-1}
\mu^{-\left[1-4(1/m+\varepsilon)\right]t}.
$$
Hence, the right hand side of \eqref{second-moment-decompsotion} is summable, and we have proved \eqref{paley-Y}.
\vspace{0.1in}

Finally \eqref{paley} implies \eqref{extinction}. Therefore, the killed branching random walk survives with positive probability.

\subsubsection{Survival without finite variance offspring distribution}
Let $R \in \N$ and define 
\begin{equation}
    \xi_R := \xi \wedge R \quad\text{ and } \quad\mu_R:= \E[\xi\wedge R]. 
\end{equation}
By monotone convergence theorem, $\lim_{R\rightarrow\infty}\mu_R = \mu.$ Therefore, we can think of a Galton-Watson tree $\T_R$ which is coupled with the original tree such that according to the Ulam-Harris representation, only first $R$ offsprings are retained. Then $\T_R$ is a Galton--Watson tree with offspring
distribution $\xi_{R}$, and, under this coupling $\T_R\subseteq \T$.
Consequently, whenever $\T_R$ survives, $\T$ automatically survives. 

We choose $R$ large enough so that $$\mu_R\in(1, \infty).$$ Thus, $\T_R$ is itself supercritical and satisfies \eqref{eq:GW1}. Next, we assume that for $\mu$, \textnormal{(A1)} and \textnormal{(A2)} holds, it is easy to verify that with
$\theta_R := \log \mu_R/\log \mu,$
$$a_R(t):= a(\theta_R t)$$ satisfies the same assumptions with $\mu_R.$

Next, set $m_R:= \theta_R m$ and further refine the choice of $R$ so that $m_R>4$. Now just from the notations above, we observe that $a_R(n/m_R) = a(n/m)$ and so \eqref{eq:Lm} is trivially satisfied.

Therefore, it follows from the previous analysis that for large enough $R$, $\T_R$ has positive probability of survival with the barrier $a(n/m).$  Thus, $\T$ has a positive probability of survival.

\section{Critical case: $m=4$}
In this section, we provide examples satisfying \textnormal{(A1)--(A3)} and ($L_m$), but exhibiting different behaviors at $m=4$. The examples indicate the second order term in assumptions on $F$ play an important role in this regime.

For simplicity, we consider deterministic binary tree. Fix $\alpha>0$, and set
$a(t) := 2^{t/\alpha}.$
Define 
$$\psi(u):= \frac{u}{\sqrt{\log(e+u)}}, \qquad u\geq 0.$$
Fix $c\in (0, 1)$ and  consider the following regularly varying  distributions supported on $[1, \infty):$
$$\Fbar_1(x)=C_1x^{-\alpha}2^{c\psi(\alpha \log_2 x)}\text{  and  }\Fbar_2(x)=C_2x^{-\alpha}2^{-c\psi(\alpha \log_2 x)}.$$
We now check the assumptions \textnormal{(A1)}, \textnormal{(A2)} and \textnormal{(A3)}. We have 
$$\log_2 \Fbar_{1, 2}(a(t))=-t\;\pm\; c\psi(t)=-t +o(t).$$ Also, for every $q>0,$ 
$$\alpha \log_2 \left(\frac{a(t/q)}{t}\right) = \frac{t}{q}-\alpha \log_2 t$$ and therefore, 
\begin{align}\label{assum-F-12}
    \log_2 \Fbar_{1,2}\left(\frac{a(t/q)}{t}\right)&= -\frac{t}{q}+ \alpha\log_2t\;\pm \; c\psi\left(\frac{t}{q}- \alpha\log_2t\right)= -\frac{t}{q}+o(t).
\end{align}
Finally, since the support of the step distribution is on $[1, \infty)$, $(L_m)$ is trivially satisfied.

\subsection{Survival at criticality.}

Here, we deal with $F_1$ and use the same embedded tree $\T_2$ as in the survival proof. Repeating the same arguments, we arrive at the second moment method for $Y_K$, which is \eqref{second-moment-decompsotion}
\begin{align}\label{second-moment-critical}
\frac{\E[Y_K^2]}{(\E[Y_K])^2}
&=
\frac{1}{2}
\sum_{t=0}^{N-1}
\frac{2^{-t}}
{\prod_{t_k\leq t}\pi_k}
+
\frac{2^{-N}}{P_K}  \leq
C
\sum_{t=0}^{N}
\frac{2^{-t}}
{\prod_{t_k\leq t}\pi_k}.
\end{align}
For $t_j\leq t<t_{j+1}$, we have 
$\prod_{t_k\leq t}\pi_k
=\prod_{k=1}^{j}\pi_k,$
and since $t_k=2^{k-1}$ imply
$\sum_{k=1}^{j}t_{k+1}
=
4t_{j}-2.$
Hence
$$
\frac{2^{-t}}
{\prod_{t_k\leq t}\pi_k}
\leq
2^{-t+t_j-1/2-c\sum_{k=0}^j\psi(t_{k+1}/4)}\leq 2^{-1/2-c\psi(t_{j+1}/4)}.$$
Thus, 
$$\sum_{t=t_j}^{t_{j+1}-1}\frac{2^{-t}}
{\prod_{t_k\leq t}\pi_k}\leq t_j 2^{-1/2-c\psi(t_{j}/2)}.$$ Since, $$\psi(t_j/2)\asymp \frac{t_j}{\sqrt{\log t_j}}$$ 
and therefore, summability in \eqref{second-moment-critical} proves the claim.

\subsection{Extinction at criticality.}
We consider the $F_2$ distribution here. Again, the proof follows the same structure as in the extinction part of the main theorem, but we will require a refined version of Lemma \ref{lem:four}, whose proof can be found in Section $4$. 
\begin{lemma}\label{lem:refined-four}
    Let
$$1=r_1<r_2<\cdots$$
be a deterministic strictly increasing sequence of positive integers. Then, for every $C>0$, we have
$$
\frac{r_2+\cdots+r_{k+1}}{r_k}
>4-\frac{C}{\log r_k}$$
for infinitely many $k$.
\end{lemma}

Now, we introduce the modified sequence of events for $t\geq 2,$ where 
\begin{equation}\label{def-E_t}
E^{\text{critical}}_t
=
\left\{
\begin{array}{l}
\exists\, 1=r_1<\cdots<r_k=t<r_{k+1}\ \text{and } v \text{ with } |v|=t,\\
X_{r_i}(v)\geq h_{r_{i+1}},\qquad 1\leq i\leq k,\\
r_2+\cdots+r_{k+1}\geq \left(4-\frac{1}{\log t}\right)t
\end{array}
\right\}.
\end{equation} 
Again existence of an infinite ray implies that $\{E^{\text{critical}}_t\}$ must occur infinitely often.
Hence, we show that they occur only finitely often, almost surely.  
\begin{lemma}
$$\sum_{t=2}^{\infty}\mathbb P(E^{\text{critical}}_t)<\infty.$$
\end{lemma}
\begin{proof} 
We first note that following \eqref{assum-F-12}, we have 
\begin{align*}
    \log_2 p_s &= \log_2 \P\left(X\geq \frac{a\left((s-1)/4\right)}{s-1}\right)  \\& =-(s-1)/4 +\alpha \log_2 (s-1)-c\psi\left(\frac{(s-1)}{4}-\alpha \log_2(s-1)\right).
\end{align*}
Since 
$$\psi\left(\frac{(s-1)}{4}-\alpha \log_2(s-1)\right) =(1+o(1))\frac{s}{4\sqrt{\log s}},$$ there exists $\gamma>0$ such that 
\begin{equation}\label{bnd-hs-critical}
p_s\leq 2^{-s/4-\gamma s/\sqrt{\log s}},    
\end{equation}
 for large enough $s.$
Using the first moment with the vertices at generation $t$ and markov inequality, we get 
$$\P(E^{\text{critical}}_t) \leq
2^t
\sum_{\substack{
k\geq2,\,
1=r_1<\cdots<r_k=t<r_{k+1}\\
r_2+\cdots+r_{k+1}\geq \left(4-\frac{1}{\log t}\right)t
}}
\prod_{i=1}^{k}p_{r_{i+1}}.$$
Therefore for large enough $t$, using \eqref{bnd-hs-critical}, we have
\begin{align}
\P(E^{\text{critical}}_t)& \leq
2^t
\sum_{\substack{
A\subset\{2,3,\ldots\}\text{ finite}\\
\sum_{s\in A}s\geq \left(4-\frac{1}{\log t}\right)t\\
t\in A
}}
\prod_{s\in A}p_s\leq 2^{t/(4 \log t)}\sum_{\substack{
A\subset\{2,3,\ldots\}\text{ finite}\\t\in A
}}
\prod_{s\in A}2^{s/4}p_s
\nonumber\\&\leq 2^{t/(4 \log t)-\gamma t/\sqrt{\log t}}\sum_{\substack{
A\subset\{2,3,\ldots\}\text{ finite}\\t\notin A
}}
\prod_{s\in A}2^{s/4}p_s\\&\leq C2^{-\gamma t/(2\sqrt{\log t})}
\prod_{s=2}^{\infty}
\left(1+2^{s/4}p_s\right).
\end{align}
Again from \eqref{bnd-hs-critical}, we have $\P(E_t^{\text{critical}})\leq C_{\text{critical}}2^{-\gamma t/(2\sqrt{\log t})}$, and hence summable. 
\end{proof}
Therefore, for the distribution $F_2,$ extinction happens almost surely at $m=4.$

\section{Proof of Technical Lemmas}
\subsection{Proof of Lemmas \ref{lem:four} and \ref{lem:refined-four}}
For $k\geq 1$, write
$$
q_k:=\frac{r_{k+1}}{r_k}
\qquad\text{and}\qquad
R_k:=\frac{r_2+\cdots+r_{k+1}}{r_k}.
$$
Since $(r_k)$ is strictly increasing, $q_k>1$ for every $k$.

We aim at proving Lemma \ref{lem:refined-four}, since Lemma \ref{lem:four} will follow as a corollary. Fix $C>0$ and suppose, for contradiction, that
$$
R_k\leq 4-\frac{C}{\log r_k}
$$
for all sufficiently large $k$. In particular, $R_k<4$ eventually. Since
$R_k\geq 1+q_k,$ we have $q_k<3$ for all sufficiently large $k$. Set
$$
x:=\limsup_{k\to\infty}q_k.
$$
Then $1\leq x\leq 3$.

We first observe that this would imply $x=2$. Suppose first that $x>1$. Fix
$\varepsilon>0$. For all sufficiently large $j$, $q_j\leq x+\varepsilon.$
Choose a subsequence $(k_\ell)$ such that $q_{k_\ell}\to x$. Using
$$
R_k
=
q_k+1+\frac{1}{q_{k-1}}
+\frac{1}{q_{k-1}q_{k-2}}
+\cdots,
$$
we obtain
$$
\lim_{\varepsilon \rightarrow 0}\limsup_{k\to\infty}R_k
\geq \lim_{\varepsilon \rightarrow 0}\left(
x+\sum_{j=0}^{\infty}\frac{1}{(x+\varepsilon)^j}\right)
=4+\frac{(x-2)^2}{x-1}.
$$
This implies that $x=2$.

Next, consider the case $x=1$. Indeed, for every $\varepsilon>0$, we have $q_j\leq1+\varepsilon$ for all
sufficiently large $j$, and therefore
$$
\liminf_{k\to\infty}R_k
\geq
1+\sum_{j=0}^{\infty}\frac{1}{(1+\varepsilon)^j}
=
1+\frac{1+\varepsilon}{\varepsilon}.
$$
Letting $\varepsilon\downarrow0$ again yields contradiction. Thus, whenever
$\limsup_{k\to\infty}q_k\neq 2,$
we have that for every $C>0$, it would imply that
$$
\frac{r_2+\cdots+r_{k+1}}{r_k}
>4-\frac{C}{\log r_k}$$
for infinitely many $k$. Finally, it remains to handle the case when $\limsup_{k\rightarrow \infty} q_k=2$. Define
$\delta_k:=4-R_k.$
By the contradiction assumption,
$$
\delta_k\geq\frac{C}{\log r_k}>0
$$
for all sufficiently large $k$.
Next, since
$R_k=q_k+R_{k-1}/q_{k-1}$, we have
$$
4-\delta_k
=
q_k+\frac{4-\delta_{k-1}}{q_{k-1}}.
$$
Consequently,
\begin{align}\label{delta-q-relation}
(q_{k-1}+\delta_{k-1})-(q_k+\delta_k)
&=
q_{k-1}-4+\frac{4}{q_{k-1}}
+\left(1-\frac{1}{q_{k-1}}\right)\delta_{k-1} \\
&=
\frac{(q_{k-1}-2)^2}{q_{k-1}}
+\left(1-\frac{1}{q_{k-1}}\right)\delta_{k-1}.
\end{align}
Both terms on the right-hand side are nonnegative, and therefore that
$q_k+\delta_k$ is eventually non-increasing. We also note that $q_k+\delta_k$ is bounded below, since $\limsup_{k\rightarrow \infty} q_k =2$ and $\delta_k>0$, by our assumptions, and hence the sequence $\{q_k+\delta_k\}_{k\in \N}$ converges. Therefore, summing both the sides of \eqref{delta-q-relation} yields
$$
\sum_k \frac{(q_k-2)^2}{q_k}<\infty\quad
\text{ and }\quad
\sum_k
\left(1-\frac{1}{q_k}\right)\delta_k<\infty.
$$
The first estimate above implies that in fact $\lim_k q_k=2.$ Hence, for all sufficiently large $k$, the second estimate would give
$$
\sum_k\delta_k<\infty.
$$

This contradicts the assumption on $\delta_k$. Indeed, since
$q_k\to2$, we have $q_k\leq3$ for all sufficiently large $k$. Thus,
for some constants $A,B>0$, $r_k\leq A3^k$
and hence
$\log r_k\leq B+k\log3.$
Therefore, there exists $C_1>0$ such that
$$
\sum \delta_k\geq C\sum_k\frac{1}{\log r_k}\geq C_1\sum_k \frac{1}{k}=\infty.
$$
This proved the lemma. Finally, the constant $4$ is sharp by taking $r_k=2^{k-1}$.

\section{Discussion}
The present work deals with the discrete-time branching random walk. A natural direction for further study is to investigate the corresponding continuous-time setting of branching Lévy processes, where the motion of each particle is governed by a Lévy process. It would be interesting to see whether an analogous phase transition can hold and to what extent the methods developed here, in particular the analysis of the infinite surviving rays, can be adapted to the continuous time setting.

\paragraph{Acknowledgements}
The author thanks Rishideep Roy and Alexander R. Watson for their helpful comments while finalizing the manuscript.

Ramkrishna Jyoti Samanta acknowledges support from the Additional Funding
Programme for Mathematical Sciences, delivered by EPSRC (EP/V521917/1)
and the Heilbronn Institute for Mathematical Research.

\bibliography{references}

\end{document}